\documentclass[12pt]{article}
\usepackage{a4wide}
\usepackage{amsthm}
\usepackage{amsfonts}
\usepackage{amssymb}
\usepackage{amsmath}
\usepackage{cite}
\usepackage{epsfig}
\usepackage{stmaryrd}
\usepackage{xcolor}
\usepackage{thm-restate}
\usepackage[hidelinks]{hyperref}
\newtheorem{theorem}{Theorem}
\newtheorem{lemma}[theorem]{Lemma}

\newcommand\HH{{\cal H}}
\newcommand\PP{{\cal P}}

\newcommand\NN{{\mathbb N}}

\newcommand\RR{{\mathbb R}}

\DeclareTextCompositeCommand{\v}{OT1}{l}{l\nobreak\hspace{-.1em}'}
\DeclareTextCompositeCommand{\v}{OT1}{t}{t\nobreak\hspace{-.1em}'\nobreak\hspace{-.15em}}

\begin{document}
\title{Solution of uniform Tur\'an's Tetrahedron Problem\thanks{The first, second and fifth authors were supported by the Alexander von Humboldt Foundation in the framework of the Alexander von Humboldt Professorship of the second author endowed by the Federal Ministry of Education and Research.}}
\author{Bart\l{}omiej Kielak\thanks{Institute of Mathematics, Leipzig University, Augustusplatz 10, 04109 Leipzig, Germany. E-mail: {\tt \{bartlomiej.kielak,xichao.shu\}@uni-leipzig.de}.}\and
        \newcounter{lth}
        \setcounter{lth}{2}
        Daniel Kr\'al'\thanks{Institute of Mathematics, Leipzig University, Augustusplatz 10, 04109 Leipzig, and Max Planck Institute for Mathematics in the Sciences, Inselstra{\ss}e 22, 04103 Leipzig, Germany. E-mail: {\tt daniel.kral@uni-leipzig.de}.}\and
	Ander Lamaison\thanks{Universidad P\'ublica de Navarra, Edificio Las Encinas, Campus de Arrosad\'ia, 31006 Pamplona, Spain. E-mail: {\tt ander.lamaison@unavarra.es}.}\and
	Hong Liu\thanks{Extremal Combinatorics and Probability Group, Institute for Basic Science, Daejeon, South Korea. Email: \texttt{hongliu@ibs.re.kr}. Supported by the Institute for Basic Science under grant IBS-R029-C4.}\and
        Xichao Shu$^\fnsymbol{lth}$\and
	Zhuo Wu\thanks{Departament de Matem\`atiques, Universitat Polit\`ecnica de Catalunya (UPC),
Carrer de Pau Gargallo 14, 08028 Barcelona, Spain. The last author acknowledges the bilateral AEI+DFG research project PCI2024-155080-2: SRC-ExCo – Structure, Randomness and Computational Methods in Extremal Combinatorics, and the PID2023-147202NB-I00 (COCOA: COntemporary COmbinatorics and its Applications), 
all funded by MICIU/AEI/10.13039/501100011033. Email: \texttt{zhuo.wu@upc.edu}} }

\date{}

\maketitle

\begin{abstract}
Tur\'an's Tetrahedron Problem asks to determine the Tur\'an density of the complete hypergraph $K_4^{(3)}$ (tetrahedron).
This problem, posed by Tur\'an in 1941, is one of the most famous problems in extremal combinatorics and
its solution would attract \$500 prize from Erd\H os.
In the 1980s, Erd\H{o}s and S\'os asked to determine Tur\'an densities of
$K_4^{(3)-}$ (broken tetrahedron) and $K_4^{(3)}$ (tetrahedron)
when edges are constrained to be uniformly distributed in the host hypergraph.
The presumably easier case of the broken tetrahedron
was solved by Glebov, Kr\'al' and Volec [Israel J. Math. 211 (2016), 349--366] and
Reiher, R\"odl and Schacht [J. Eur. Math. Soc. 20 (2018), 1139--1159].
We solve the tetrahedron case
by proving that
the uniform Tur\'an density of $K_4^{(3)}$ is equal to $1/2$;
this confirms that R\"odl's lower bound construction from 1986 is optimal.
\end{abstract}

\maketitle

\section{Introduction}
\label{sec:intro}

Tur\'an problems are among the most fundamental problems in extremal combinatorics;
they ask to determine the minimum density, called the \emph{Tur\'an density}, that guarantees the existence of a given substructure.
Specifically, the 85-year-old \emph{Tur\'an's Tetrahedron Problem},
which asks to determine the Tur\'an density of the complete $3$-uniform $4$-vertex hypergraph $K_4^{(3)}$,
is generally considered to be one of the most intriguing problems in the area.
While Tur\'an densities of graphs are well understood due to the classical work of Erd\H os and Stone~\cite{ErdS46} from the 1940s (see also~\cite{ErdS66}),
Tur\'an densities of hypergraphs have turned out to be extremely challenging:
Erd\H os offered \$500 for determining the Tur\'an density of
any single complete $k$-uniform hypergraph with at least $k+1$ vertices and
offered \$1\,000 for determining the Tur\'an density of all complete $k$-uniform hypergraphs for $k\ge 3$.
Neither of the prizes has been claimed.
We refer to the surveys by Keevash~\cite{Kee11} and Sidorenko~\cite{Sid95} for results on Tur\'an densities of hypergraphs,
including progress towards Tur\'an's Tetrahedron Problem.

Almost all constructions that are known or conjectured to be extremal in the Tur\'an density setting
have edges distributed in a highly non-uniform way.
This led Erd\H os and S\'os~\cite{ErdS82,Erd90} to introduce the notion of the \emph{uniform Tur\'an density},
which is the minimum density that guarantees the existence of a given hypergraph
under an additional assumption that the edges are distributed uniformly in the host hypergraph.
Formally, an $n$-vertex hypergraph is $(d,\varepsilon)$-uniformly dense 
if every subset of at least $\varepsilon n$ vertices has density at least $d$; for brevity, we will simply call such a hypergraph $(d,\varepsilon)$-dense.
The uniform Tur\'an density of a $k$-uniform hypergraph $H$ is the infimum over all $d$ such that
there exists $\varepsilon>0$ such that every sufficiently large $(d,\varepsilon)$-dense hypergraph contains $H$ as a subhypergraph.

In the 1980s,
Erd\H os and S\'os asked about determining the uniform Tur\'an densities of
the $3$-uniform hypergraphs $K_4^{(3)-}$ (broken tetrahedron) and $K_4^{(3)}$ (tetrahedron).
Until a decade ago, there was very little progress on the uniform Tur\'an densities of hypergraphs.
The uniform Tur\'an density of $K_4^{(3)-}$ was determined by Glebov, Volec and the second author~\cite{GleKV16}
using the flag algebra method of Razborov~\cite{Raz07}, and
a direct combinatorial argument using the hypergraph regularity method
was given by Reiher, R\"odl and Schacht~\cite{ReiRS18a}.
This application of the hypergraph regularity catalyzed progress in the area and many results quickly followed.
In particular,
Reiher, R\"odl and Schacht~\cite{ReiRS18} classified $3$-uniform hypergraphs with uniform Tur\'an density equal to $0$, and
showed that the uniform Tur\'an density of every $3$-uniform hypergraph is at least $1/27$ unless it is equal to $0$;
the value of $1/27$ was shown to be tight in~\cite{GarKL24}.
The uniform Tur\'an density of non-trivial tight cycles was determined in~\cite{BucCKMM23}, and
that of generalized stars in~\cite{LamW24}, largely extending the result concerning $K_4^{(3)-}$.
Additional recent results on the uniform Tur\'an density of $3$-uniform hypergraphs
can be found in~\cite{CheSXX,GarIKLXX,KinPSS25,KinSS24,LiLWZ23}.
For a broader overview, we refer the reader to the survey by Reiher~\cite{Rei20} on the topic,
which also discusses results on stronger notions of uniform density such as those considered in~\cite{ReiRS16,ReiRS18b,ReiRS18c}.

Despite the significant insights into uniform Tur\'an densities of hypergraphs that have been recently gained,
determining the uniform Tur\'an density of $K_4^{(3)}$,
which has generally been considered to be the most central open problem concerning the notion,
remained elusive.
A particular reason for the substantial difference between the broken tetrahedron and the tetrahedron cases
is the presence of triangles in the link graphs of vertices that may form edges of the host hypergraph (abstractly speaking,
the difference is that
when $K_4^{(3)-}$ is viewed as a $2$-dimensional simplicial complex, it  is contractible while $K_4^{(3)}$ is not).
In a certain sense,
the hypergraph regularity permits exploiting the absence of such triangles to reduce the analysis 
from the hypergraph setting to graphs and this is not possible in the case of the tetrahedron.

A non-trivial lower bound on the uniform Tur\'an density of $K_4^{(3)}$
was provided by R\"odl~\cite{Rod86} in 1986 (this lower bound construction
is presented at the beginning of Section~\ref{sec:palette}), and
it was generally believed that this construction is optimal.
Our main result solves the problem posed by Erd\H os and S\'os on determining the uniform Tur\'an density of $K_4^{(3)}$ and
confirms that R\"odl's 40-year-old lower bound construction is indeed optimal.

\begin{restatable}[Uniform Tur\'an density of the Tetrahedron]{theorem}{thmmain}
\label{thm:main}
The uniform Tur\'an density of $K_4^{(3)}$ is equal to $1/2$.
\end{restatable}

The rest of the paper is devoted to proving Theorem~\ref{thm:main}.
We provide a detailed overview of the proof in Section~\ref{sec:overview};
we hope that this overview can be accessible to any combinatorialist
while covering all key steps of the proof to the point that the proof becomes reconstructible.
In the next two paragraphs,
we briefly mention the main techniques used in the proof of Theorem~\ref{thm:main}.

The two proofs of the broken tetrahedron case employ different arguments:
the proof given in~\cite{GleKV16} is based on Razborov's flag algebra method~\cite{Raz07},
which revolutionized extremal combinatorics (see e.g.~\cite{BalCL22,FalPVV23,HatHKNR12,HatHKNR13,MubR21,Raz08,Raz10}), and
the proof given in~\cite{ReiRS18} is based on the hypergraph regularity,
one of the cornerstone techniques of modern hypergraph theory.
We remark that
these arguments yield an upper bound of $0.529$ on the uniform Tur\'an density of the tetrahedron~\cite[Table 1]{BalCL22s}.
The hypergraph regularity argument of Reiher, R\"odl and Schacht
paved the way for a general framework for analyzing uniform Tur\'an densities,
which was formalized in~\cite{Rei20} and
which yielded a long line of results that were mentioned earlier.
This framework ultimately led to a classification of uniform Tur\'an densities by palette constructions~\cite{Lam24}.

The proof of Theorem~\ref{thm:main} utilizes all these developments,
which alone however turned insufficient to resolve the tetrahedron case, and
a new key ingredient, the reduction technique described in Section~\ref{sec:excl}.
The technique,
which we believe to be new in this context,
yields that certain patterns cannot appear in extremal constructions.
However, instead of excluding the pattern directly (which seems impossible),
we exclude every dense configuration containing the pattern, which in turn excludes the pattern.
While it may seem puzzling at first sight that such a type of argument could work in principle,
it turned out to be surprisingly powerful and
led to breaking the longstanding resistance of the uniform Tur\'an Tetrahedron problem.

\section{Palettes}
\label{sec:palette}

A key concept in our argument is that of a palette, which we introduce in this section.
To motivate this concept,
we first recall R\"odl's lower bound construction on the uniform Tur\'an density of $K_4^{(3)}$ from~\cite{Rod86}.
The construction proceeds as follows:
consider $n$ vertices $v_1,\ldots,v_n$,
color each pair independently red or blue, each with probability $1/2$, and
include $\{v_i,v_j,v_k\}$, $i<j<k$, as an edge
if the colors of the pairs $v_iv_j$ and $v_iv_k$ differ.
The resulting hypergraph does not contain $K_4^{(3)}$ as a subhypergraph and
for every $d\in (0,1/2)$ and $\varepsilon>0$
it is $(d,\varepsilon)$-dense with positive probability when $n$ is sufficiently large.
This shows that the uniform Tur\'an density of $K_4^{(3)}$ is at least $1/2$.

We now present a generalization of this construction,
which nowadays forms a standard lower bound framework known as palette constructions.
A \emph{palette} $\PP$ is a pair $(\Gamma,T)$ with $T\subseteq \Gamma^3$;
we refer to the elements of $\Gamma$ as \emph{colors} and the elements of $T$ as (feasible) \emph{triples}.
For brevity, we write $xyz$ for the ordered triple $(x,y,z)$.
If $xyz$ is a feasible triple,
we refer to $x$ as the \emph{left} color, $y$ as the \emph{top} color and $z$ as the \emph{right} color of the triple.
The \emph{density} of a palette $\PP=(\Gamma,T)$, which is denoted by $d(\PP)$, is $|T|/|\Gamma|^3$.
We will see that the palette corresponding to R\"odl's construction is the palette $\PP_4=(\Gamma,T)$
where $\Gamma=\{b,r\}$ and $T=\{brb,brr,rbb,rbr\}$;
note that $d(\PP_4)=1/2$.

We say that an $n$-vertex hypergraph $H$ is \emph{$\PP$-colorable}
if its vertices can be ordered as $v_1,\ldots,v_n$ and
there exists an assignment $f:\binom{V(H)}{2}\to\Gamma$ of colors contained in $\Gamma$ to the (unordered) pairs of vertices of $H$ such that
it holds \[f(\{v_i,v_j\})f(\{v_i,v_k\})f(\{v_j,v_k\})\in T\]
for every edge $\{v_i,v_j,v_k\}$ of $H$ with $1\le i<j<k\le n$.
Note that $K_4^{(3)}$ is not $\PP_4$-colorable.
More generally, $K_4^{(3)}$ is $\PP$-colorable for a palette $\PP=(\Gamma,T)$ if and only if
there exist not necessarily distinct colors $u$, $v$, $w$, $x$, $y$ and $z$ in $\Gamma$ such that
$uxv$, $uzy$, $xzw$ and $vyw$ are feasible triples (see Figure~\ref{fig:K4}).

\begin{figure}
\begin{center}
\epsfbox{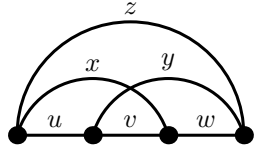}
\end{center}
\caption{An assignment of colors $u$, $v$, $w$, $x$, $y$ and $z$ to the pairs of vertices witnessing that
         $K_4^{(3)}$ is $\PP$-colorable if the palette $\PP$ contains the triples $uxv$, $uzy$, $xzw$ and $vyw$.}
\label{fig:K4}
\end{figure}

Now, we consider the following random construction of an $n$-vertex hypergraph $H_n$ with vertices $v_1,\ldots,v_n$:
assign a color $c\in\Gamma$ to each pair $v_iv_j$, $1\le i<j\le n$, uniformly at random and
include $\{v_i,v_j,v_k\}$, $1\le i<j<k\le n$, as an edge if $xyz\in T$
where $x$ is the color assigned to $v_iv_j$, $y$ is the color assigned to $v_iv_k$ and $z$ is the color assigned to $v_jv_k$.
Observe that $H_n$ is $\PP$-colorable and so it does not contain any hypergraph that is not $\PP$-colorable.
Using standard concentration inequalities,
it is possible to show for every $d\in\left(0,d(\PP)\right)$ and $\varepsilon>0$ that
$H_n$ is $(d,\varepsilon)$-dense with positive probability when $n$ is sufficiently large.
It follows that if a hypergraph $H$ is not $\PP$-colorable,
then the uniform Tur\'an density of $H$ is at least $d(\PP)$.

The third author~\cite{Lam24} showed that palette constructions always
provide tight lower bounds on the uniform Tur\'an density.

\begin{theorem}
\label{thm:palette}
The uniform Tur\'an density of a $3$-uniform hypergraph $H$
is equal to the supremum of the densities of the palettes $\PP$ such that $H$ is not $\PP$-colorable.
\end{theorem}

\noindent Theorem~\ref{thm:palette} expedited progress in the area as
it permits replacing arguments using the hypergraph regularity with more direct combinatorial arguments concerning palettes.
In particular, it was essential in~\cite{LamW24} to determine the uniform Tur\'an density of generalized stars.
We refer to~\cite{KinPSS25,KraKLT25} for additional applications.
Theorem~\ref{thm:palette} is also one of the cornerstones of the proof of Theorem~\ref{thm:main}.

\section{Proof overview}
\label{sec:overview}

In this section, we provide a high level overview of the proof of Theorem~\ref{thm:main}.
We believe that the proof is conceptually simple although particular steps are involved and
some require computational assistance.
In particular,
we are convinced that
an interested reader would be able to fill in all details based on the information provided in this section.

Theorem~\ref{thm:palette} implies that, in order to prove Theorem~\ref{thm:main},
it suffices to rule out the existence of a palette $\PP=(\Gamma,T)$ with density strictly larger than $1/2$ such that
$K_4^{(3)}$ is not $\PP$-colorable.
In the rest of the section,
we sketch an argument ruling out the existence of such palette $\PP$.
By considering a sufficiently large set of colors,
we can assume that each triple $T$ contains three distinct colors from $\Gamma$ (palettes that
have this property will be referred to as \emph{simple}).
In addition,
a standard symmetrization argument yields (see Lemma~\ref{lm:palette}) that
we can also assume that
there exists $\delta>0$ such that
each color $c$ is contained in at least $\left(\frac{1}{2}+\delta\right)3|\Gamma|^2$ triples of $T$.

\begin{figure}
\begin{center}
\epsfbox{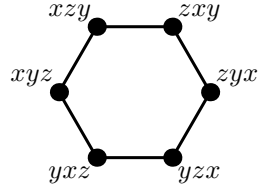}
\end{center}
\caption{Assignment of ordered triples of colors $x$, $y$ and $z$ to the vertices of a hexagon.}
\label{fig:hex0}
\end{figure}

\begin{figure}
\begin{center}
\epsfbox{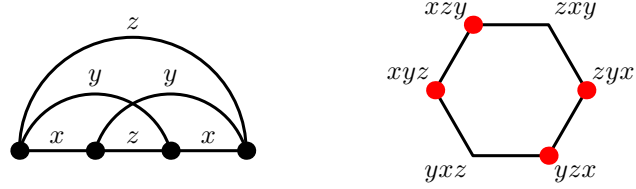}
\end{center}
\caption{An assignment of colors $x$, $y$ and $z$ to the pairs of vertices witnessing that
         $K_4^{(3)}$ is $\PP$-colorable if the palette $\PP$ contains the triples $xyz$, $xzy$, $yzx$ and $zyx$.
	 The four triples are visualized as vertices of the hexagon in the right part of the figure.}
\label{fig:hex-}
\end{figure}

\begin{figure}
\begin{center}
\epsfbox{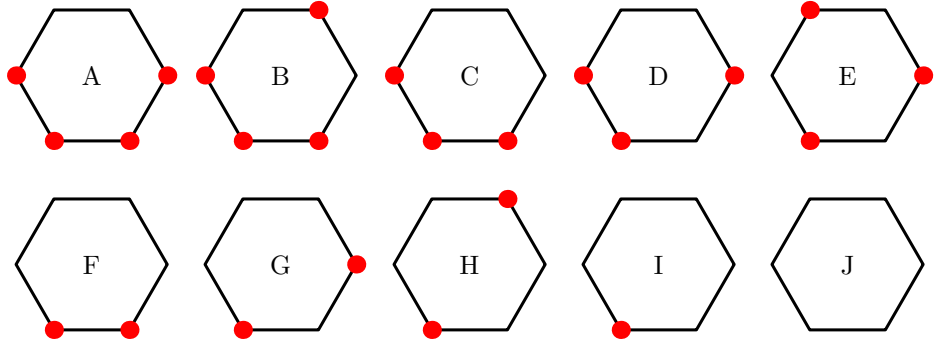}
\end{center}
\caption{The ten types of unordered triples denoted by A, B, $\ldots$, J.}
\label{fig:hex10}
\end{figure}

We now classify unordered triples of colors, i.e., three-element subsets of $\Gamma$.
We will visualize the ordered triples of colors contained in $T$ by highlighting the corresponding vertices of a hexagon, according to the correspondence depicted in Figure~\ref{fig:hex0};
note that the triples are listed in a way that the swaps between the left and the top colors and
the swaps between the right and the top colors alternate.
Observe that if $T$ contained the triples $xyz$, $xzy$, $yzx$ and $zyx$ for some three colors $x$, $y$ and $z$,
then $K_4$ would be $\PP$-colorable (see Figure~\ref{fig:hex-}).
In particular,
no hexagon visualizing the triples contained in $T$ has two opposite consecutive pairs of vertices highlighted.
It follows that each unordered triple (up to a symmetry of a hexagon)
is one of the ten types depicted in Figure~\ref{fig:hex10};
we will refer to these types as A, B, $\ldots$, J following the notation given in Figure~\ref{fig:hex10}.

We associate a palette $\PP=(\Gamma,T)$ with the complete $3$-uniform hypergraph $\HH$ on the vertex set $\Gamma$ such that
each edge of $\HH$ is labelled by one of A, B, $\ldots$, J depending on the type of the triple of its vertices.
The edges of $\HH$ have weights assigned based on the number of ordered triples represented by an edge that are contained in $\PP$,
i.e., the edges of type A and B have weight $4$,
the edges of type C, D and E have weight $3$,
the edges of type F, G and H have weight $2$,
the edges of type I have weight $1$ and
the edges of type J have weight $0$.
We will write $a, b, \ldots, j$ for the proportion of the edges of the types A, B, $\ldots$, J in $\HH$, respectively.
The definition of these quantities yields that
\begin{equation}
a+b+c+d+e+f+g+h+i+j=1,\label{eq:sum1}
\end{equation}
and the assumption on the density of the palette $\PP$ implies that
\[\frac{4(a+b) + 3(c+d+e) + 2(f+g+h) + i}{6}\ge\frac{1}{2}+\delta,\]
which is equivalent to
\begin{equation}
4(a+b) + 3(c+d+e) + 2(f+g+h) + i \ge 3+6\delta.\label{eq:low3a}
\end{equation}

We now introduce our main reduction technique.
Let $H_0$ be a $3$-uniform hypergraph on $k_0$ vertices with its edges labelled by A, B, $\ldots$, J, and fix $k\in\NN$, $k\ge 2$.
Suppose that
the following holds for every $3$-uniform hypergraph $H$ on $k_0+k$ vertices
whose edges are labelled A, B, $\ldots$, J and that satisfies the following conditions:
$H_0$ is induced by the $k_0$ vertices, and
the average of the weights of the edges containing one of the $k_0$ vertices and
two of the remaining $k$ vertices is strictly larger than $3$.
Then any choice of ordered triples on $k_0+k$ colors that
is consistent with the types of the edges of $H$ colors $K_4^{(3)}$,
i.e., there exist (not necessarily distinct) colors $a,b,c,d,e,f$ such that $adb$, $afe$, $bec$, and $dfc$ are among the ordered triples.
This technique is illustrated on an example in Figure~\ref{fig:excl}.
Note that this claim can be computationally verified for every fixed $H_0$ and an integer $k\in\NN$ (although
if $H_0$ or $k$ is large, the computation may take very substantial time).
A hypergraph $H_0$ that satisfies the above claim for some $k\in\NN$
will be referred to as \emph{excluded}.

We remark that 
we actually work with more fine-grained notions of a $k$-excluded motif and a $(k,w)$-excluded motif instead of
the notion of an excluded hypergraph in Section~\ref{sec:excl};
the notion of a $(k,w)$-excluded motif involves a weighted average over the vertices of $H_0$,
which makes it possible for the reduction to work with smaller values of $k$.
However, we believe that the presentation of the proof overview
is more accessible without introducing these more involved notions
although we may be technically inaccurate as the weighted average is indeed needed in some of our arguments.

\begin{figure}
\begin{center}
\epsfbox{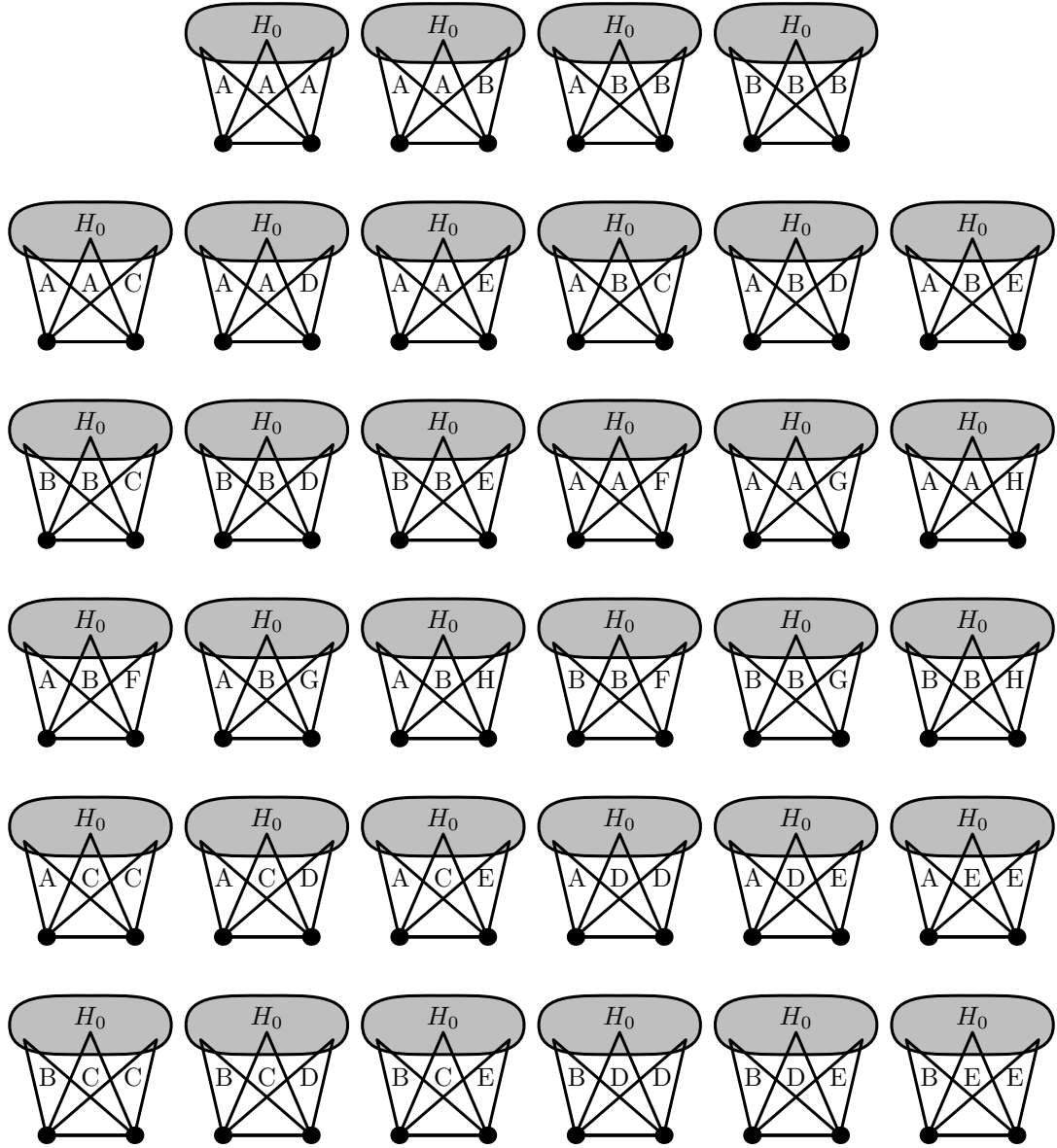}
\end{center}
\caption{Illustration of the definition of an excluded hypergraph.
         The figure lists all 34 hypergraphs such that
	 the average weight of edges containing one of the top three vertices and
	 the two bottom vertices is strictly larger than $3$.
         If $K_4^{(3)}$ is $\PP$-colorable
	 for every palette $\PP$ that is associated with a hypergraph symmetric to one of the 34 depicted hypergraphs,
	 then the 3-vertex hypergraph $H_0$ is excluded.
	 This corresponds to the case $k_0=3$ and $k=2$ in the definition
	 of an excluded hypergraph given in Section~\ref{sec:overview}.}
\label{fig:excl}
\end{figure}

We now argue that if a hypergraph $H_0$ is excluded, then it cannot appear in $\HH$.
Suppose that $H_0$ satisfies the above claim for $k\in\NN$.
If $\HH$ contained $H_0$, then there would exist a choice of $k$ vertices such that
the average weighted degree of the $k_0$ vertices of $H_0$ with respect to the $k$ vertices is larger than $3\binom{k}{2}$;
the existence of such $k$ vertices follows from the fact that
the minimum weighted degree of $\HH$ is at least $3\left(\frac{1}{2}+\delta\right)n^2$
where $n$ is the number of vertices of $\HH$ (note that $n$ can be made large while keeping $\delta>0$ fixed).
However, the existence of such $k$ vertices in $\HH$ is impossible as $H_0$ is excluded and
$K_4^{(3)}$ is not $\PP$-colorable.

We computationally verified that the $3$-vertex hypergraph $H_0$ formed by a single edge of type E is excluded;
it follows that the $3$-vertex hypergraph formed by a single edge of type B is also excluded.
In addition to a substantial structural impact on a palette $\PP$,
the absence of the edges of types B and E yields that $b=e=0$ in \eqref{eq:low3a},
which implies that
\begin{equation}
4a+3c+3d+2f+2g+2h+i \ge 3+6\delta.\label{eq:low3b}
\end{equation}
We also computationally verified that
each hypergraph obtained from the Fano plane by labeling each edge either C or D is excluded (the argument
used weights and the absence of the edges of types B and E as discussed in Section~\ref{sec:excl}).
Hence, each hypergraph obtained from the Fano plane by labeling each edge A, C or D is excluded.
Since the Tur\'an density of the Fano plane,
which was determined by de Caen and F\"uredi~\cite{DecF00} (we also refer to~\cite{KeeS05} for an exact result),
is equal to $3/4$,
we obtain that
\begin{equation}
a+c+d\le 3/4+o(1).\label{eq:Fano}
\end{equation}
The $o(1)$ term corresponds to lower-order terms in $|\Gamma|$ and
it can be made arbitrarily small while keeping $\delta>0$ fixed.

\begin{figure}
\begin{center}
\epsfbox{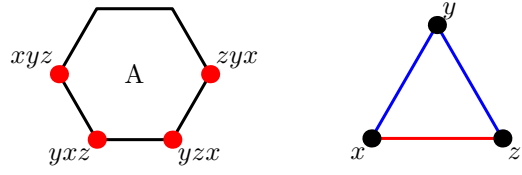}
\end{center}
\caption{Illustration of the definition of the legs and the base of an edge of type A.}
\label{fig:redblue}
\end{figure}

The final ingredient of the proof is the inequality
\begin{equation}
a-2j\le 1/4+o(1),\label{eq:flag}
\end{equation}
which we now describe how to obtain.
We start with constructing an auxiliary hypergraph $\HH'$ which is derived from $\HH$.
First observe that each edge of type A consists of four triples such that
\textbf{either} two pairs differ by the swap of the left and top colors and one pair by the swap of the right and top colors \textbf{or}
two pairs differ by the swap of the right and top colors and one pair by the swap of the left and top colors.
The two pairs that differ by the swap of the same kind will be referred to as \emph{legs} of the edge of type A and
the remaining one will be referred to as its \emph{base}.
The leg pairs will be visualized as blue pairs and the base pairs as red, which is illustrated in Figure~\ref{fig:redblue}.
It is easy to show (see Lemma~\ref{lm:redblue}) that
if a pair of vertices is both the leg of an edge of type A and the base of another edge of type A, then $K_4^{(3)}$ is $\PP$-colorable.

\begin{figure}
\begin{center}
\epsfbox{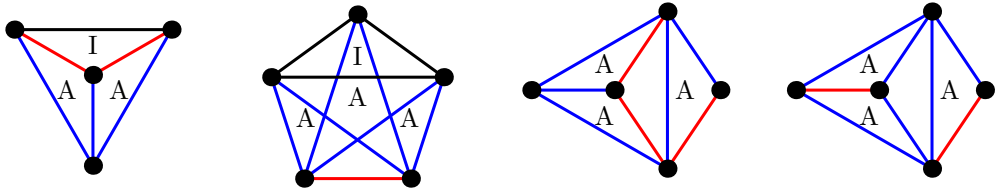}
\end{center}
\caption{Excluded hypergraphs in $\HH$ when the pairs of vertices forming a leg of an edge of type A are colored blue and
         other pairs are colored red. The black pairs can be either blue or red.}
\label{fig:AIJ}
\end{figure}

\begin{figure}
\begin{center}
\epsfbox{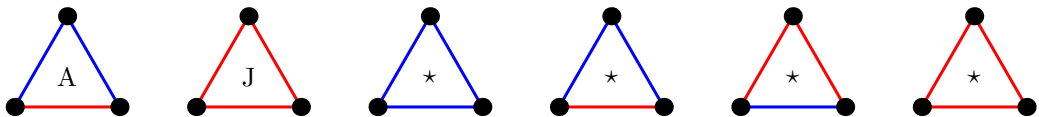}
\end{center}
\caption{Possible types of edges in the hypergraph $\HH'$.}
\label{fig:auxHH}
\end{figure}

The hypergraph $\HH'$ is obtained from $\HH$ as follows:
if a pair of vertices is a leg of an edge of type A, then color it blue, and color the pair red otherwise.
Note that the base of any edge of type A is always colored red (as it cannot appear as a leg of any edge of type A).
Using the same argument as for the hypergraphs excluded from appearing in $\HH$,
we computationally verified that
none of the four hypergraphs depicted in Figure~\ref{fig:AIJ} can appear in $\HH$
when the vertex pairs are colored blue and red as described above.
The hypergraph $\HH'$ is obtained by relabelling any edge of type C, D, F, G, H or I to the type~$\star$ and
also relabelling any edge of type J with at least one blue pair to the type~$\star$.
Hence, the hypergraph $\HH'$ is a complete $3$-uniform hypergraph,
each edge is of one of the types A, J and $\star$, and
the pairs of vertices of the hypergraph $\HH'$ are colored either red or blue.
Note that each edge of $\HH'$ is of one of the six types depicted in Figure~\ref{fig:auxHH}.
Assuming that none of the hypergraphs depicted in Figure~\ref{fig:AIJ} is present (before relabelling the edges),
the inequality \eqref{eq:flag} can be established using a flag algebra argument (specific technical details
are given in Section~\ref{sec:flag}) applied to $\HH'$.
In particular,
the inequality holds in a stronger form when $j$ is the proportion of edges of type J in $\HH'$,
which can be smaller than that in $\HH$.

To complete the proof,
we consider the inequality that
is obtained by summing twice \eqref{eq:sum1}, \eqref{eq:Fano} and \eqref{eq:flag} (recall that $b=e=0$):
\[4a+3c+3d+2f+2g+2h+2i\le 3+o(1).\]
Since $i\ge 0$, this inequality contradicts \eqref{eq:low3b}.
Hence, there is no palette $\PP$ with density larger than $1/2$ such that $K_4^{(3)}$ is not $\PP$-colorable.

\section{Preliminaries}
\label{sec:prelim}

In this section,
we apply standard arguments concerning palettes to establish all assumptions on a considered palette $\PP$ that
were presented in Section~\ref{sec:overview}.
Before doing so, we need to fix some notation;
throughout the paper, we use $[k]$ to denote the set of the first $k$ positive integers and
$\RR^+$ to denote the set of all positive reals.
We also need to introduce an additional definition.
The \emph{Lagrangian} $L(\PP)$ of a palette $\PP=(\Gamma,T)$ is
\[\max_{p:\Gamma\to [0,1], \sum p=1}\sum_{xyz\in T}p(x)p(y)p(z)\]
where the maximum is taken over all probability distributions $p$ on $\Gamma$,
i.e., functions $p:\Gamma\to [0,1]$ such that $\sum_{x\in \Gamma}p(x)=1$.

The next lemma is based on a routine argument from mathematical analysis and
it is a continuous version of Zykov's symmetrization~\cite{Zyk49},
which can be used to provide an alternative proof of Lemma~\ref{lm:palette}.
We include a short proof for completeness.

\begin{lemma}
\label{lm:sympal}
Let $\PP=(\Gamma,T)$ be a palette.
There exists a probability distribution $p$ on $\Gamma$ such that
it holds for every color $x\in\Gamma$ that
\[\sum_{xyz\in T}p(y)p(z)+\sum_{yxz\in T}p(y)p(z)+\sum_{yzx\in T}p(y)p(z)=3L(\PP)\]
unless $p(x)=0$.
\end{lemma}

\begin{proof}
Fix a palette $\PP=(\Gamma,T)$ and
consider a function $F:[0,1]^\Gamma\to\RR$ defined as
\[F(s)=\sum_{xyz\in T} s_x s_y s_z.\]
Note that $L(\PP)$ is the maximum of $F$ over all non-negative vectors $s\in [0,1]^\Gamma$
whose entries sum to one.
Let $s\in [0,1]^\Gamma$ be such a vector at which $F$ attains its maximum.
The Lagrange Multiplier Theorem implies that there exists $C\in\RR$ such that
\[\frac{\partial}{\partial s_x}F(s)=C\]
for every $x\in\Gamma$ such that $s_x\not=0$ (note that this statement is also true
if $s$ has a single non-zero entry).
Observe that
\[\frac{\partial}{\partial s_x}F(s) = \sum_{xyz\in T}s_ys_z+\sum_{yxz\in T}s_ys_z+\sum_{yzx\in T}s_ys_z.\]
It now follows that
\[3L(\PP) = \sum_{xyz\in T}s_xs_ys_z+\sum_{yxz\in T}s_ys_xs_z+\sum_{yzx\in T}s_ys_zs_x
           = \sum_{x\in\Gamma} s_x \frac{\partial}{\partial s_x}F(s) = C \sum_{x\in\Gamma} s_x = C.\]
Hence, the statement of the lemma holds for $p(x)=s_x$ as
the formula displayed in the statement is exactly $\frac{\partial}{\partial s_x}F(s)$,
which is equal to $C=3L(\PP)$ for any color $x$ such that $s_x\not=0$.
\end{proof}

We are now ready to state the main lemma of this section.

\begin{lemma}
\label{lm:palette}
If the uniform Tur\'an density of $K_4^{(3)}$ were strictly larger than $1/2$,
then there would exist an integer $n_0$ and a real $d>1/2$ such that
for every $n\ge n_0$, there would be a palette $\PP$ with $n$ colors and the following properties:
\begin{itemize}
\item $K_4^{(3)}$ is not $\PP$-colorable,
\item $\PP$ is simple, i.e., each feasible triple consists of three distinct colors, and
\item each color of $\PP$ is contained in at least $3dn^2+36n$ feasible triples.
\end{itemize}
\end{lemma}

\begin{proof}
Suppose that the uniform Tur\'an density of $K_4^{(3)}$ is larger than $1/2$.
By Theorem~\ref{thm:palette},
there exists a palette $\PP_0=(\Gamma_0,T_0)$ with density $d_0>1/2$ such that $K_4^{(3)}$ is not $\PP_0$-colorable.
Note that the Lagrangian $L(\PP_0)$ of $\PP_0$ is at least $d_0$ (as witnessed by the uniform distribution on $\Gamma_0$).
Let $p_0$ be the probability distribution of $\Gamma_0$ as in Lemma~\ref{lm:sympal}.

For every $n\in\NN$, we construct a palette $\PP_n$ with $n$ colors as follows.
Fix integers $k_x\in\NN$, $x\in\Gamma_0$, such that
\[
k_x\in\{\lfloor p_0(x)n\rfloor,\lceil p_0(x)n\rceil\}\qquad\mbox{and}\qquad
\sum_{x\in\Gamma_0}k_x=n.
\]
Such a choice is always possible.
Note that $k_x=0$ whenever $p_0(x)=0$.
For each color $x\in\Gamma_0$,
the palette $\PP_n$ contains $k_x$ clones of the color $x$, which will be denoted by $x_1,\ldots,x_{k_x}$.
Note that the palette $\PP_n$ has exactly $n$ colors.
A triple $x_iy_jz_k$ is a feasible triple in the palette $\PP_n$
if $xyz$ is a feasible triple in $\PP_0$ and the colors $x_i$, $y_j$ and $z_k$ are distinct.

Set $d=d_0/2+1/4$ and note that $1/2<d<d_0$.
We now verify that the palette $\PP_n$ has the three properties given in the statement of the lemma (for the chosen value of $d$)
when $n$ is sufficiently large.
If $K_4^{(3)}$ were $\PP_n$-colorable,
then assigning the color $x$ of $\PP_0$ to any pair of the vertices of $K_4^{(3)}$ that is assigned a color $x_i$ of $\PP_n$ (and
keeping the order of the vertices the same) would witness that $K_4^{(3)}$ is $\PP_0$-colorable.
Hence, $K_4^{(3)}$ is not $\PP_n$-colorable.
On the other hand, the construction process guarantees that $\PP_n$ is simple.

It remains to analyze the last property given in the statement of the lemma.
Observe that each color $z$ of $\PP_n$ is contained in at least
\[3d_0n^2-3|\Gamma_0|n-3n-3n\]
feasible triples
where the first term corresponds to the number of triples that would contain $z$ in the ideal fractional case,
the second term is an upper bound on the number of such triples lost because of the choices of $k_x$ (which is $3n$ per each color of $\Gamma_0$),
the third term is an upper bound on the number of triples removed because they contain the color $z$ at least twice, and
the last term is an upper bound on the number of triples removed because they contain a color different from $z$ twice.
Set $n_0\in\NN$ so that it holds that
\[3d_0n^2-3|\Gamma_0|n-6n\ge 3dn^2+36n\]
for every $n\ge n_0$ (recall that $d<d_0$).
We conclude for any $n\ge n_0$ that
the palette $\PP_n$ also has the third property given in the statement of the lemma.
\end{proof}

\section{Excluded motifs}
\label{sec:excl}

In this section, we analyze the hypergraph that is derived from a palette $\PP$ in the way described in Section~\ref{sec:overview}.
Recall that if $\PP$ is a simple palette such that $K_4^{(3)}$ is not $\PP$-colorable,
then any three colors of $\PP$ give rise to one of the ten types of configurations depicted in Figure~\ref{fig:hex10}.
Hence, we can associate a simple palette $\PP=(\Gamma,T)$ with the complete $3$-uniform hypergraph $\HH(\PP)$ such that
the vertex set of $\HH(\PP)$ is $\Gamma$ and
an edge $e$ is labeled based on the feasible triples containing the three colors forming the edge $e$.
In addition, for each edge of type A, we color the two pairs of colors forming its legs blue (the base and
the legs of an edge of type A are defined as in Section~\ref{sec:overview}).
We color the remaining pairs of colors red.
Hence, the hypergraph $\HH(\PP)$ is an extension of the hypergraph $\HH$ considered in Section~\ref{sec:overview}.
As mentioned in Section~\ref{sec:overview} and proven in the next lemma,
no pair of colors can be the base of an edge of type A and a leg of another edge of type A.
In particular, each edge of type A has its base colored red and the two legs colored blue.

\begin{lemma}
\label{lm:redblue}
Let $\PP$ be a simple palette.
If there are four colors $x$, $y$, $z$ and $z'$ such that
the triple $\{x,y,z\}$ is an edge of type A in $\HH(\PP)$ and $xy$ is its leg, and
the triple $\{x,y,z'\}$ is an edge of type A in $\HH(\PP)$ and $xy$ is its base,
then $K_4^{(3)}$ is $\PP$-colorable.
\end{lemma}

\begin{figure}
\begin{center}
\epsfbox{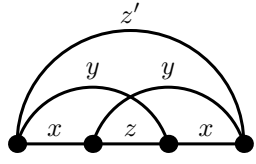}
\end{center}
\caption{An assignment of colors to the pairs of vertices of $K_4^{(3)}$
         witnessing that $K_4^{(3)}$ is $\PP$-colorable for any palette $\PP$ that
	 contains four colors $x$, $y$, $z$ and $z'$ such that
	 all triples $xyz$, $zyx$, $xz'y$ and $yz'x$ are feasible.}
\label{fig:rbK4}
\end{figure}

\begin{proof}
Fix a simple palette $\PP$ and colors $x$, $y$, $z$ and $z'$ as in the statement of the lemma.
By symmetry, we may assume that $xz$ is the base of the edge $\{x,y,z\}$.
It follows that the following triples are feasible:
$xyz$ and $zyx$ (because $xz$ is the base of the edge $\{x,y,z\}$) and
$xz'y$ and $yz'x$ (because $xy$ is the base of the edge $\{x,y,z'\}$).
An assignment of the colors to the pairs of vertices of $K_4^{(3)}$ witnessing that
$K_4^{(3)}$ is $\PP$-colorable can be found in Figure~\ref{fig:rbK4}.
\end{proof}

In Section~\ref{sec:flag}, we will need the following auxiliary lemma.
Since its proof is analogous to that of Lemma~\ref{lm:redblue}, we state and prove it now.

\begin{lemma}
\label{lm:leftright}
Let $\PP=(\Gamma,T)$ be a simple palette.
If there exist colors $x$ and $y$ such that
there are feasible triples that differ by swapping $x$ and $y$ as the left and top colors and
there are feasible triples that differ by swapping $x$ and $y$ as the right and top colors,
i.e.~there exist colors $z$ and $z'$ such that $xyz, yxz, z'xy, z'yx\in T$,
then $K_4^{(3)}$ is $\PP$-colorable.
\end{lemma}

\begin{figure}
\begin{center}
\epsfbox{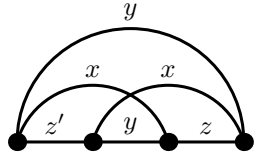}
\end{center}
\caption{An assignment of colors to the pairs of vertices of $K_4^{(3)}$
         witnessing that $K_4^{(3)}$ is $\PP$-colorable for any palette $\PP$ that
	 contains four colors $x$, $y$, $z$ and $z'$ such that
	 all triples $xyz$, $yxz$, $z'xy$ and $z'yx$ are feasible.}
\label{fig:lrK4}
\end{figure}

\begin{proof}
Fix a simple palette $\PP$ and colors $x$, $y$, $z$ and $z'$ as in the statement of the lemma.
An assignment of the colors to the pairs of vertices of $K_4^{(3)}$ witnessing that
$K_4^{(3)}$ is $\PP$-colorable can be found in Figure~\ref{fig:lrK4}.
\end{proof}

For presenting our reduction techniques, we need the following auxiliary lemma.

\begin{lemma}
\label{lm:avg}
Let $k_0,k\in\NN$, $k\ge 2$, and let $w:[k_0]\to\RR^+$.
The following holds for every simple palette $\PP=(\Gamma,T)$ such that
$\PP$ has at least $k_0+k$ colors, i.e., $|\Gamma|\ge k_0+k$, and
each color of $\PP$ is contained in more than $3|\Gamma|^2/2+6(k_0-1)|\Gamma|$ feasible triples:
for every $k_0$-element subset $X=\{x_1,\ldots,x_{k_0}\}\subseteq\Gamma$,
there exist a $k$-element subset $Y\subseteq\Gamma\setminus X$ such that
\[
\sum_{i\in [k_0]}w(i)\left(\left|T\cap\left(\{x_i\}\times Y^2\right)\right| +
                           \left|T\cap\left( Y\times\{x_i\}\times Y\right)\right| +
                           \left|T\cap\left( Y^2\times\{x_i\}\right)\right|\right)\]
is larger than $\frac{3}{2}|Y|\left(|Y|-1\right)\sum\limits_{i\in [k_0]}w(i)$.
\end{lemma}

\begin{proof}
Fix $k_0,k\in\NN$, a simple palette $\PP=(\Gamma,T)$ and a $k_0$-element subset $X\subseteq\Gamma$ as in the statement of the lemma.
Set $Y_0=\Gamma\setminus X$.
Consider a color $x\in X$.
The color $x$ is in at most $6|\Gamma|$ feasible triples with any fixed color $x'\in X$ and
so the number of feasible triples containing $x$ and another color of $X$ is at most $6(k_0-1)|\Gamma|$.
It follows the number of feasible triples containing $x$ and two colors from $Y_0$
is larger than
\[\frac{3}{2}|\Gamma|^2\ge \frac{3}{2}|Y_0|\left(|Y_0|-1\right).\]
Hence, when a $k$-element subset $Y\subseteq Y_0$ is chosen uniformly at random,
the expected number of feasible triples containing the color $x$ and two colors from $Y$ is larger than $\frac{3}{2}k(k-1)$.
Since this holds for every color $x\in X$,
the linearity of expectation implies that
when a $k$-element subset $Y\subseteq Y_0$ is chosen uniformly at random,
the expectation of the displayed sum in the statement of the lemma 
is larger than
\[\frac{3}{2}k(k-1)\sum\limits_{i\in [k_0]}w(i).\]
In particular, there exists a $k$-element subset $Y\subseteq Y_0$ such that
the displayed inequality holds.
\end{proof}

\subsection{Simple reduction technique}

We analyze configurations that cannot appear in a palette $\PP$ that
has the properties given in Lemma~\ref{lm:palette}.
Such configurations consist of colors and their triples, i.e. they can be viewed as palettes.
To emphasize their different role in our argument, we will refer to them as motifs.
Formally, a \emph{motif} is a pair $(\gamma,t)$ such that $t\subseteq\gamma^3$,
i.e. $(\gamma,t)$ forms a palette.
We say that a palette $\PP=(\Gamma,T)$ \emph{contains} a motif $(\gamma,t)$
if there exists an injective map $\varphi:\gamma\to\Gamma$ such that
$\varphi(x)\varphi(y)\varphi(z)\in T$ for every $xyz\in t$.
Notations defined earlier for palettes are also used in the setting of motifs,
in particular,
we write $\HH(\gamma,t)$ for the hypergraph $\HH(\PP)$ and
we say that $K_4^{(3)}$ is $(\gamma,t)$-colorable if it is $\PP$-colorable
where $\PP=(\gamma,t)$.
Finally, two motifs $(\gamma,t)$ and $(\gamma',t')$ are \emph{isomorphic}
if they differ by renaming the colors and possibly swapping the roles of left and right colors in all feasible triples,
i.e. $(\gamma,t)$ and $(\gamma',t')$ are isomorphic
if there exists a bijection $f:\gamma\to\gamma'$ such that
\textbf{either} $xyz\in t$ iff $f(x)f(y)f(z)\in t$ for all $x,y,z\in\gamma$ \textbf{or}
$xyz\in t$ iff $f(z)f(y)f(x)\in t$ for all $x,y,z\in\gamma$.

Consider $k\ge 2$. A motif $(\gamma,t)$ is \emph{$k$-excluded}
if $K_4^{(3)}$ is $\PP$-colorable
by any simple palette $\PP=(\Gamma,T)$ with $\gamma\subseteq\Gamma$, $|\Gamma|=|\gamma|+k$ and $t\subseteq T$
such that
the palette $\PP$ contains more than $3|\gamma|k(k-1)/2$ feasible triples containing one color from $\gamma$ and two colors from $\Gamma\setminus\gamma$.
Note that this definition can be viewed as a special case of the definition of an excluded hypergraph given in Section~\ref{sec:overview} in the motif setting:
a hypergraph $H$ is excluded (in the sense given in Section~\ref{sec:overview})
if and only if
every motif $(\gamma,t)$ consistent with $H$ is $k$-excluded for some $k$.
We remark that it is not hard to show that
if a motif $(\gamma,t)$ is $k$-excluded, then it is also $k'$-excluded for every $k'\ge k$.
The following is a key fact for our arguments:
checking whether a motif $(\gamma,t)$ is $k$-excluded can be done computationally
since there are only finitely many palettes $\PP$ with the properties given in the definition of a $k$-excluded motif.

The importance of $k$-excluded motifs is captured by the next lemma.

\begin{lemma}
\label{lm:excl}
Let $(\gamma,t)$ be a $k$-excluded motif for some $k\ge 2$.
Every simple palette $\PP=(\Gamma,T)$ such that
\begin{itemize}
\item $K_4^{(3)}$ is not $\PP$-colorable,
\item $\PP$ has at least $|\gamma|+k$ colors, i.e., $|\Gamma|\ge |\gamma|+k$, and
\item each color is contained in more than $3|\Gamma|^2/2+6(|\gamma|-1)|\Gamma|$ feasible triples,
\end{itemize}
does not contain the motif $(\gamma,t)$.
\end{lemma}

\begin{proof}
Fix a simple palette $\PP=(\Gamma,T)$ with the properties described in the statement of the lemma.
Let $(\gamma,t)$ be a $k$-excluded motif, $k\ge 2$.
Assume for contradiction that
the palette $\PP$ contains the motif $(\gamma,t)$,
i.e., there exists an injective map $\varphi:\gamma\to\Gamma$ such that
$\varphi(x)\varphi(y)\varphi(z)\in T$ for every $xyz\in t$.
Apply Lemma~\ref{lm:avg} with $k_0=|\gamma|$, $k$, $w$ constantly equal to one, $\PP$ and $X=\varphi(\gamma)$
to get a $k$-element subset $Y\subseteq\Gamma\setminus X$ with the properties described in the statement of Lemma~\ref{lm:avg}.
In particular, the palette $\PP$ restricted to the colors $X\cup Y$
has the properties given in the definition of a $k$-excluded motif,
which implies that $K_4^{(3)}$ is $\PP$-colorable contrary to our assumption.
\end{proof}

The following lemma has been verified using Gurobi Optimizer (version 11.0.3)~\cite{gurobi} and
the code is available as an ancillary file on arXiv.

\begin{lemma}
\label{lm:excl-E}
The motif $(\gamma,t)$ where $\gamma=\{x,y,z\}$ and $t=\{xyz,yzx,zxy\}$ is $4$-excluded.
\end{lemma}

\subsection{Main reduction technique}

We now present a more involved notion of an excluded motif.
A motif $(\gamma,t)$ is \emph{$(k,w)$-excluded}
for an integer $k\ge 2$ and a weight function $w:\gamma\to\RR^+$
if $K_4^{(3)}$ is $\PP$-colorable
by any simple palette $\PP=(\Gamma,T)$ such that
$\gamma\subseteq\Gamma$, $|\Gamma|=|\gamma|+k$, $t\subseteq T$,
the hypergraph $\HH(\PP)$ has no edge of type B or E, and
\begin{align*}
\sum_{x\in\gamma}w(x)\Big(
  & \left|T\cap\left(\{x\}\times \left(\Gamma\setminus\gamma\right)^2\right)\right| +
  \left|T\cap\left( \left(\Gamma\setminus\gamma\right)\times\{x\}\times \left(\Gamma\setminus\gamma\right)\right)\right| +\\
  & \left|T\cap\left( \left(\Gamma\setminus\gamma\right)^2\times\{x\}\right)\right|\Big)
  >\frac{3}{2}\left|\Gamma\setminus\gamma\right|\left(\left|\Gamma\setminus\gamma\right|-1\right)\sum_{x\in\gamma}w(x).
\end{align*}			   
In other words,
the weighted average of edges of $\HH(\PP)$ containing one color from $\gamma$ and two colors from $\Gamma\setminus\gamma$ is larger than 3.
Again, a key fact for our arguments is the following:
checking whether a motif $(\gamma,t)$ is $(k,w)$-excluded can be done computationally (when $k$ and $w$ are fixed)
since there are only finitely many palettes $\PP$ with the properties given in the definition.
We would also like to point out the following subtle property:
in principle,
a $(k,w)$-excluded motif where $w$ is the constant function equal to one need not be $k$-excluded
since the definition of an $(k,w)$-excluded motif assumes the absence of edges of type B and E.

While the notion is formally weaker, the analogue of Lemma~\ref{lm:excl} still holds.

\begin{lemma}
\label{lm:excl-BE}
Let $(\gamma,t)$ be a $(k,w)$-excluded motif for some $k\ge 2$ and $w:\gamma\to\RR^+$.
Every simple palette $\PP=(\Gamma,T)$ such that
\begin{itemize}
\item $K_4^{(3)}$ is not $\PP$-colorable,
\item $\PP$ has at least $\max\{|\gamma|+k,7\}$ colors, i.e., $|\Gamma|\ge\max\{|\gamma|+k,7\}$, and
\item each color is contained in more than $3|\Gamma|^2/2+6\max\{|\gamma|-1,2\}|\Gamma|$ feasible triples,
\end{itemize}
does not contain the motif $(\gamma,t)$.
\end{lemma}

\begin{proof}
Fix a simple palette $\PP=(\Gamma,T)$ with the properties described in the statement of the lemma.
Let $(\gamma,t)$ be a $(k,w)$-excluded motif.
Assume for contradiction that
the palette $\PP$ contains the motif $(\gamma,t)$,
i.e., there exists an injective map $\varphi:\gamma\to\Gamma$ such that
$\varphi(x)\varphi(y)\varphi(z)\in T$ for every $xyz\in t$.
Apply Lemma~\ref{lm:avg} with $k_0=|\gamma|$, $k$, $w$, $\PP$ and $X=\varphi(\gamma)$
to get a $k$-element subset $Y\subseteq\Gamma\setminus X$ with the properties described in the statement of Lemma~\ref{lm:avg}.
Since the hypergraph $\HH(\PP)$ does not contain an edge of type B or E by Lemmas~\ref{lm:excl} and~\ref{lm:excl-E},
the restriction of the palette $\PP$ to the colors $X\cup Y$
has the properties given in the definition of a $(k,w)$-excluded motif.
This implies that $K_4^{(3)}$ is $\PP$-colorable contrary to our assumption.
\end{proof}

We verified the following lemma using Gurobi Optimizer (version 11.0.3)~\cite{gurobi} and
the code is available as an ancillary file on arXiv.

\begin{lemma}
\label{lm:Fano}
Let $(\gamma,t)$ be a motif with $|\gamma|=7$ such that
the edges of $\HH(\gamma,t)$ are of types A, C, D and J only and
the edges of the types A, C and D form the Fano plane.
If $K_4^{(3)}$ is not $(\gamma,t)$-colorable,
then $(\gamma,t)$ is isomorphic to one of the motifs depicted in Figure~\ref{fig:Fano}.
\end{lemma}

\begin{figure}
\begin{center}
\epsfbox{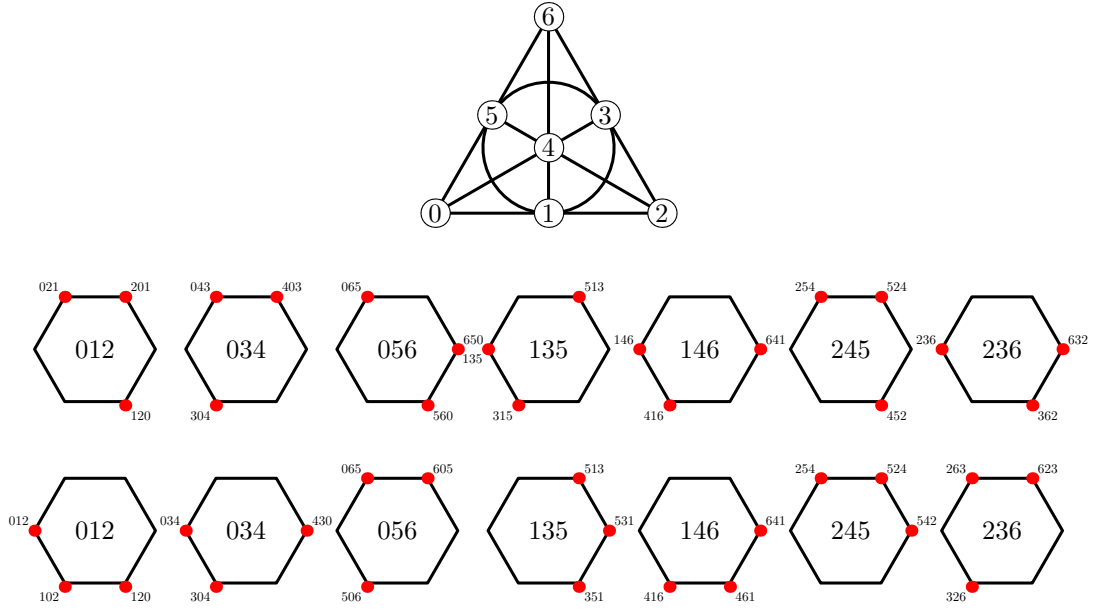}
\end{center}
\caption{Visualization of the two motifs from Lemma~\ref{lm:Fano}.
         The notation for the colors forming the Fano plane is visualized at the top and
	 each of the two motifs is presented on one of the two lines below the Fano plane
	 using the notation from Figure~\ref{fig:hex0}.}
\label{fig:Fano}
\end{figure}

The following lemma has been verified using Gurobi Optimizer (version 11.0.3)~\cite{gurobi} and
the code is available as an ancillary file on arXiv.

\begin{lemma}
\label{lm:excl-Fano}
The top motif depicted in Figure~\ref{fig:Fano} is $(4,w)$-excluded
where $w$ is the constant function equal to one, and
the bottom motif depicted in Figure~\ref{fig:Fano} is $(4,w)$-excluded
where $w$ is the function defined as follows:
\[w(0)=0.5, w(1)=1.5, w(2)=1, w(3)=1, w(4)=1, w(5)=2.5\mbox{ and }w(6)=1.\]
\end{lemma}

\subsection{Conclusion of section}

We summarize the findings of this section in the next lemma.

\begin{lemma}
\label{lm:excl-main}
Let $\PP=(\Gamma,T)$ be a simple palette such that
\begin{itemize}
\item $K_4^{(3)}$ is not $\PP$-colorable,
\item $\PP$ has at least $11$ colors, and
\item each color is contained in more than $3|\Gamma|^2/2+36|\Gamma|$ feasible triples.
\end{itemize}
The hypergraph $\HH(\PP)$ has no edge of type B or E and
the spanning subhypergraph of $\HH(\PP)$ consisting of all edges of type A, C and D
does not contain the Fano plane as a subhypergraph.
\end{lemma}

\begin{proof}
Fix a palette $\PP=(\Gamma,T)$ as in the statement of the lemma.
The hypergraph $\HH(\PP)$ has no edge of type B or E by Lemmas~\ref{lm:excl} and~\ref{lm:excl-E} (note that
the motif described in the statement of Lemma~\ref{lm:excl-E} is contained in each edge of type B and E).
Similarly, Lemmas~\ref{lm:excl-BE}, \ref{lm:Fano} and~\ref{lm:excl-Fano} imply that
the spanning subhypergraph of $\HH(\PP)$ consisting of all edges of type A, C and D
does not contain the Fano plane as a subhypergraph.
\end{proof}

\section{Final estimate}
\label{sec:flag}

In this section, we discuss the proof of the inequality \eqref{eq:flag} from Section~\ref{sec:overview}.
To prove the inequality \eqref{eq:flag}, we need to establish that some additional motifs are excluded.
As in the case of Lemmas~\ref{lm:excl-E} and~\ref{lm:excl-Fano},
the following lemma has been verified using Gurobi Optimizer (version 11.0.3)~\cite{gurobi}, and
the code is available as an ancillary file on arXiv.
Note that each motif $(\gamma,t)$ such that
$K_4^{(3)}$ is not $(\gamma,t)$-colorable and
the hypergraph $\HH(\gamma,t)$ is one of the four hypergraphs depicted in Figure~\ref{fig:AIJ}
is isomorphic to one of the eight motifs depicted in Figure~\ref{fig:AIJ0}.

\begin{figure}
\begin{center}
\epsfbox{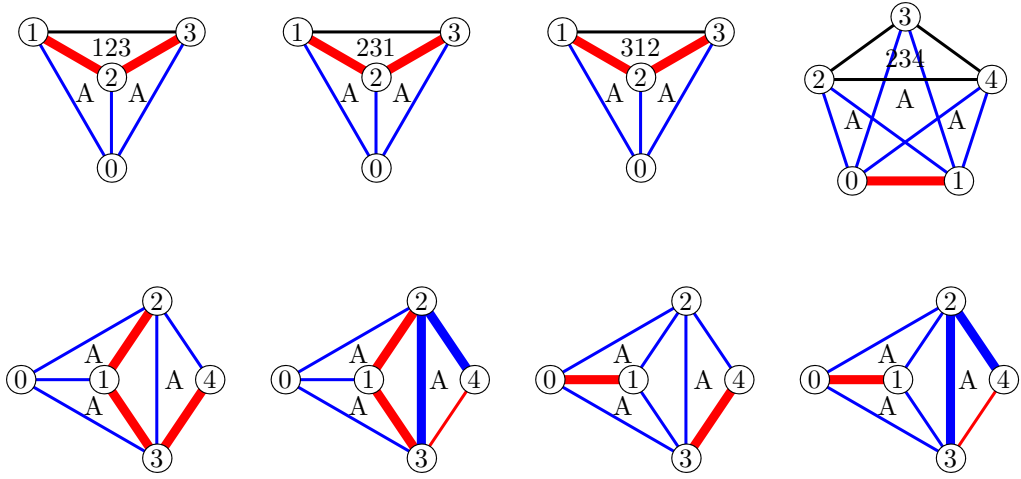}
\end{center}
\caption{The eight excluded motifs from Lemma~\ref{lm:excl-AIJ}.
         The thickness of the segments determines the following:
         thin segments correspond to pairs where the left and the top colors are swapped in the edge of type A and
	 thick segments correspond to pairs where the right and the top colors are swapped.
	 Note that if an edge is of type A, then the thin and thick segments uniquely determine
	 which four triples of the three colors forming the edge are feasible.
         Finally, the top edges of the first three motifs in the first row and
	 the right edge in the last motif in the first row have type I and
	 their single feasible triple is the one displayed in the figure.
         The numbering of vertices is used in Table~\ref{tab:AIJ} to define weight functions.}
\label{fig:AIJ0}
\end{figure}

\begin{lemma}
\label{lm:excl-AIJ}
Let $(\gamma,t)$ be one of the eight motifs given in Figure~\ref{fig:AIJ0}.
The motif $(\gamma,t)$ is $(k,w)$-excluded where the values of $k$ and $w$ are given in Table~\ref{tab:AIJ}.
\end{lemma}

\begin{table}
\begin{center}
\begin{tabular}{|c|ccccc|}
\hline
 $k$ & $w(0)$ & $w(1)$ & $w(2)$ & $w(3)$ & $w(4)$ \\
\hline
 $5$ & $1$ & $1$ & $1.25$ & $1$ & \\
 $4$ & $1$ & $1$ & $2$ & $1$ & \\
 $5$ & $1$ & $1$ & $1.25$ & $1$ & \\
\hline
 $5$ & $1$ & $1$ & $1.25$ & $1.25$ & $1.25$ \\
\hline
 $4$ & $1$ & $1$ & $1$ & $1$ & $1$ \\
 $4$ & $1.5$ & $0.8$ & $1.8$ & $1.65$ & $1$ \\
\hline
 $4$ & $1$ & $1$ & $1$ & $1$ & $1$ \\
 $4$ & $1$ & $1$ & $2$ & $2$ & $2$ \\
\hline 
\end{tabular}
\end{center}
\caption{The values of $k$ and $w$ from the statement of Lemma~\ref{lm:excl-AIJ}.
         The motifs are listed in the same order as in Figure~\ref{fig:AIJ0} and
	 they are horizontally grouped based on their correspondence to the hypergraphs depicted in Figure~\ref{fig:AIJ}.
	 The numbering of vertices is given in Figure~\ref{fig:AIJ0}.}
\label{tab:AIJ}
\end{table}

We now associate a simple palette such that $K_4^{(3)}$ is not $\PP$-colorable
with a complete $3$-uniform hypergraph $\HH'(\PP)$ as in Section~\ref{sec:overview}.
The hypergraph $\HH'(\PP)$ is obtained from $\HH(\PP)$
by relabelling any edge of type C, D, F, G, H or I to the type~$\star$ and
also relabelling any edge of type J with at least one blue pair to the type~$\star$.
Lemmas~\ref{lm:excl-BE} and~\ref{lm:excl-AIJ} yield the following.

\begin{figure}
\begin{center}
\epsfbox{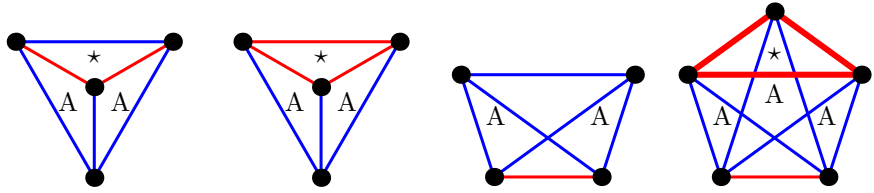}
\end{center}
\caption{Subgraphs that cannot appear in the hypergraph $\HH'(\PP)$
         when a palette $\PP$ has the properties given in Lemma~\ref{lm:palette}.
	 In the right most subgraph, the types are specified
	 for the three triangles formed by the bottom two vertices and one of the top vertices and
         the triangle formed by the top three vertices (which is visualized by having all segments drawn thick).}
\label{fig:excl-AK}
\end{figure}

\begin{lemma}
\label{lm:excl-AK}
Let $\PP=(\Gamma,T)$ be a simple palette such that
\begin{itemize}
\item $K_4^{(3)}$ is not $\PP$-colorable,
\item $\PP$ has at least $10$ colors, and
\item each color is contained in more than $3|\Gamma|^2/2+36|\Gamma|$ feasible triples.
\end{itemize}
The hypergraph $\HH'(\PP)$ does not contain any of the four hypergraphs depicted in Figure~\ref{fig:excl-AK}.
\end{lemma}

\begin{proof}
Fix a simple palette $\PP=(\Gamma,T)$ with the properties given in the statement of the lemma.
Lemmas~\ref{lm:excl-BE} and~\ref{lm:excl-AIJ} imply that 
$\PP$ contains none of the eight motifs given in Figure~\ref{fig:AIJ0}.
Recall that Lemma~\ref{lm:leftright} implies that
there is no pair of colors $x$ and $y$ such that
there are feasible triples that differ by swapping $x$ and $y$ as the left and top colors and
there are feasible triples that differ by swapping $x$ and $y$ as the right and top colors.
In particular, if two edges of type A of $\HH'(\PP)$ share a leg,
then this leg is \textbf{either} a left-top swap in both edges \textbf{or} a right-top swap in both edges, and
if they share the base,
then the base is \textbf{either} a left-top swap in both edges \textbf{or} a right-top swap in both edges.

Consider two edges $\{x,y,z\}$ and $\{x,y,z'\}$ of type A of $\HH'(\PP)$ such that
the pairs $xz$ and $xz'$ are their bases, i.e., they are red.
By symmetry, we can assume that the pair $xy$ is a left-top swap in both edges.
The pair $zz'$ must be red as $\PP$ contains neither of the two left motifs in the bottom row in Figure~\ref{fig:AIJ0}.
Since $\PP$ contains neither of the three left motifs in the top row of Figure~\ref{fig:AIJ0},
there is no feasible triple containing the colors $x$, $z$ and $z'$.
It follows that the edge $\{x,z,z'\}$ is an edge of type J in $\HH'(\PP)$.
In particular, $\HH'(\PP)$ does not contain either of the first two hypergraphs depicted in Figure~\ref{fig:excl-AK}.

Next consider two edges $\{x,y,z\}$ and $\{x,y,z'\}$ of type A in $\HH'(\PP)$ such that
the pair $xy$ is the base in both of them.
By symmetry, we can assume that the pair $xy$ is a right-top swap in both edges.
The pair $zz'$ cannot be blue as $\PP$ contains neither of the two right motifs in the bottom row in Figure~\ref{fig:AIJ0}.
Hence, $\HH'(\PP)$ does not contain the third hypergraph depicted in Figure~\ref{fig:excl-AK}.
In addition, consider another edge $\{x,y,z''\}$ of type A; note that the pair $xy$ must be its base.
Since $\HH'(\PP)$ does not contain the third hypergraph depicted in Figure~\ref{fig:excl-AK},
all pairs $zz'$, $zz''$ and $z'z''$ are red.
Finally, since $\PP$ does not contain the rightmost motif in the top row of Figure~\ref{fig:AIJ0},
there is no feasible triple containing the colors $z$, $z'$ and $z''$ and so 
the edge $\{z,z',z''\}$ is an edge of type J in $\HH'(\PP)$.
In particular, $\HH'(\PP)$ does not contain the fourth hypergraph depicted in Figure~\ref{fig:excl-AK}.
\end{proof}

We say that a complete $3$-uniform hypergraph $H$ is \emph{consistent} if
\begin{itemize}
\item each edge of $H$ is of one the three types A, J and $\star$,
\item each pair of vertices of $H$ is colored with blue or red,
\item each edge of type A has one red and two blue pairs,
\item all three pairs in each edge of type J are red, and
\item $H$ does not contain any of the four hypergraphs depicted in Figure~\ref{fig:excl-AK}.
\end{itemize}
Note that if the hypergraph $\HH'(\PP)$ arises from a palette $\PP$ that satisfies the assumptions of Lemma~\ref{lm:excl-AK},
then the hypergraph $\HH'(\PP)$ is consistent.
The following lemma was proven using the flag algebra method;
the certificate along with a code to verify it is available as an ancillary file on arXiv.

\begin{lemma}
\label{lm:flag}
For every $\varepsilon>0$, there exists $n_0$ such that
every $n$-vertex complete $3$-uniform consistent hypergraph $H$ with $n\ge n_0$
satisfies that
\[m_A\le\left(\frac{1}{4}+\varepsilon\right)\binom{n}{3}+2m_J\]
where $m_A$ is the number of edges of type A and $m_J$ is the number of edges of type J in $H$.
\end{lemma}

\section{Main result}
\label{sec:main}

We are now ready to prove Theorem~\ref{thm:main}, which we restate here for convenience.

\thmmain*

\begin{proof}
Suppose for contradiction that the uniform Tur\'an density of $K_4^{(3)}$ is strictly larger than $1/2$ and
let $n_0$ and $d>1/2$ be as given by Lemma~\ref{lm:palette}.
Set $\varepsilon=d-1/2$.
Since the Tur\'an density of the Fano plane is $3/4$~\cite{DecF00,KeeS05},
there exists $n_1$ such that any $n$-vertex $3$-uniform hypergraph $H$ with $n\ge n_1$ that
does not contain the Fano plane as a subhypergraph
has at most $\left(\frac{3}{4}+\varepsilon\right)\binom{n}{3}$ edges.
Further, let $n_2$ be the integer $n_0$ from Lemma~\ref{lm:flag} applied with $\varepsilon$.

Set $n=\max\{n_0,n_1,n_2,11\}$ and let $\PP$ be a simple palette with $n$ colors
with the properties given in Lemma~\ref{lm:palette}.
Let $m_A,\ldots,m_J$ be the number of edges of types A, \dots, J of the hypergraph $\HH(\PP)$.
Since each color is contained in at least $3dn^2+36n$ feasible triples,
the palette $\PP$ has at least $dn^3+12n^2$ feasible triples and so it holds that
\begin{equation}
4(m_A+m_B)+3(m_C+m_D+m_E)+2(m_F+m_G+m_H)+m_I\ge dn^3+12n^2\ge 6d\binom{n}{3}.
\label{eq:A-J}
\end{equation}
By Lemma~\ref{lm:excl-main},
the hypergraph $\HH(\PP)$ has no edge of type B or E and so $m_B=m_E=0$,
which combines with \eqref{eq:A-J} to
\begin{equation}
4m_A+3(m_C+m_D)+2(m_F+m_G+m_H)+m_I\ge 6d\binom{n}{3}.
\label{eq:ACDFHIJ}
\end{equation}
In addition, Lemma~\ref{lm:excl-main} yields that
the spanning subhypergraph of $\HH(\PP)$ consisting of all edges of type A, C and D
does not contain the Fano plane as a subhypergraph,
which implies that
\begin{equation}
m_A+m_C+m_D\le \left(\frac{3}{4}+\varepsilon\right)\binom{n}{3}.
\label{eq:ACD}
\end{equation}
Finally, we obtain from Lemma~\ref{lm:flag} applied to $\HH'(\PP)$ (note that each edge of type J in $\HH'(\PP)$
has type J in $\HH(\PP)$) the following:
\begin{equation}
m_A-2m_J\le \left(\frac{1}{4}+\varepsilon\right)\binom{n}{3}.
\label{eq:AJ}
\end{equation}
Since each edge of $\HH(\PP)$ is one of the types A, C, D, F, G, H, I and J,
we obtain that
\begin{equation}
m_A+m_C+m_D+m_F+m_G+m_H+m_I+m_J=\binom{n}{3}.
\label{eq:ACDEFHIJsum}
\end{equation}
We obtain the following as the sum twice \eqref{eq:ACDEFHIJsum}, \eqref{eq:ACD} and \eqref{eq:AJ}:
\[
4m_A+3(m_C+m_D)+2(m_F+m_G+m_H)+2m_I\le \left(3+2\varepsilon\right)\binom{n}{3},
\]
which contradicts \eqref{eq:ACDFHIJ} as $m_I\ge 0$ and $6d=3+6\varepsilon$.
\end{proof}

\section*{Acknowledgements}

None of the arguments presented in this paper was obtained with AI assistance.
For full transparency, we disclose that when we expected the inequality \eqref{eq:flag} to hold,
we prompted ChatGPT-5.6 and received a flag algebra proof of the inequality
assuming that two left hypergraphs depicted in Figure~\ref{fig:AIJ} are excluded.
We then identified additional excluded hypergraphs and
obtained our own flag algebra proof of the inequality;
our flag algebra computation also involves a smaller number of configurations.

The authors would like to thank Haoran Luo for pointing out a flaw in the statement of Lemma~\ref{lm:avg},
which resulted from discretizing the limit version of our argument for the purpose of presenting it in the paper, and
for bringing to our attention that we incorrectly copied one of the weights from our code to the statement of Lemma~\ref{lm:excl-Fano}.

The same result was obtained in~\cite{Buc26}.
The second author of this manuscript shared with the author of~\cite{Buc26} on August 18, 2026, that
we had obtained a computer-assisted proof of the result that does not contain AI-assisted arguments.
The author of~\cite{Buc26} disclosed that he had also been working independently on the problem and
expressed his already existing intention to feed his strategies into an LLM. ChatGPT-5.6,
which he subsequently used, was unable to complete the proof of the result.
On September 6, 2026, two days after its release, ChatGPT-6 managed to complete the proof,
using an additional prompt added to the previous conversation with ChatGPT-5.6.
On September 8, 2026, after verifying and simplifying the argument,
the author of~\cite{Buc26} informed us about having obtained an alternative computer-assisted proof and
shared a draft with us---at that point, our manuscript had already been submitted to arXiv,
but had not yet become publicly available or been shared with the author of~\cite{Buc26}.
The two proofs use different approaches.

A statement that verbatimly mirrors the previous paragraph also appears in~\cite{Buc26}.


\bibliographystyle{bibstyle}
\bibliography{turank4}
\end{document}